\documentclass[11pt]{amsart}
\usepackage{amssymb,amscd,amsthm,amsmath,color}
\usepackage{fullpage}
\newcommand{\PP}{\mathbb{P}}
\newcommand{\OO}{\mathcal{O}}

\newcommand{\aA}{\mathcal{A}}

\newcommand{\Hom}{\textsf{Hom}}

\newcommand{\End}{\operatorname{End}}
\newcommand{\Mor}{\operatorname{Mor}}
\newcommand{\Ext}{\operatorname{Ext}}
\newcommand{\M}{\operatorname{M}_e}
\newcommand{\GL}{\operatorname{GL}}

\theoremstyle{plain}
\newtheorem{lemma}{Lemma}[section]
\newtheorem*{theorem*}{Theorem}
\newtheorem*{lemma*}{Lemma}
\newtheorem*{proposition*}{Proposition}
\newtheorem*{conjecture*}{Conjecture}
\newtheorem*{corollary*}{Corollary}
\newtheorem*{problem*}{Problem}
\newtheorem{theorem}[lemma]{Theorem}

\newtheorem{corollary}[lemma]{Corollary}
\newtheorem{proposition}[lemma]{Proposition}

\newtheorem{question}[lemma]{Question}

\theoremstyle{definition}
\newtheorem{definition}[lemma]{Definition}

\newtheorem{remark}[lemma]{Remark}

\begin{document}

\title{Normal bundles of rational curves in projective space}
\author[I. Coskun]{Izzet Coskun}
\address{Department of Mathematics, Statistics and CS \\University of Illinois at Chicago, Chicago, IL 60607}
\email{coskun@math.uic.edu}

\author[E. Riedl]{Eric Riedl}
\email{ebriedl@uic.edu}

\thanks{During the preparation of this article the first  author was partially supported by the NSF CAREER grant DMS-0950951535 and the NSF grant DMS 1500031; and the second author was partially supported by an NSF RTG grant DMS-1246844.}

\begin{abstract}
Let $b_{\bullet}$ be a sequence of integers $1 < b_1 \leq b_2 \leq \dots \leq b_{n-1}$. Let  $\M(b_{\bullet})$ be the space parameterizing nondegenerate, immersed, rational curves of degree $e$ in $\PP^n$ such that the normal bundle has the splitting type $\bigoplus_{i=1}^{n-1}\OO(e+b_i)$. When $n=3$, celebrated results of Eisenbud, Van de Ven, Ghione and Sacchiero show that $\M(b_{\bullet})$ is irreducible of the expected dimension. We show that when $n \geq 5$, these loci are generally reducible with components of higher than the expected dimension. We give examples where the number of components grows linearly with $n$. These generalize an example of Alzati and Re.
\end{abstract}

\maketitle

\section{Introduction}
Rational curves play a central role in the birational and arithmetic geometry of projective varieties. Consequently, understanding the geometry of the space of rational curves is of fundamental importance. The local structure of this space is governed by the normal bundle. In this paper, we study the dimensions and irreducible components of the loci in the space of rational curves in $\PP^n$ parameterizing curves whose normal bundles have a specified splitting type. We work over an algebraically closed field of characteristic zero.

We first set some notation. Let $f: \PP^1 \rightarrow \PP^n$ be a nondegenerate, unramified, birational map of degree $e$.  Then the normal bundle $N_f$ defined by $$0 \longrightarrow T_{\PP^1}\stackrel{\text{d} f}{ \longrightarrow} f^* T_{\PP^n} \longrightarrow N_f \longrightarrow 0$$ is a vector bundle of rank $n-1$ and degree $e(n+1) -2$.  By Grothendieck's theorem, $N_f$ is isomorphic to a direct sum of line bundles. Let $\Mor_e(\PP^1, \PP^n)$ denote the morphism scheme parameterizing degree $e$ morphisms $f: \PP^1 \rightarrow \PP^n$. Let $b_{\bullet}$ denote an increasing  sequence of integers $$1 < b_1 \leq b_2 \leq \dots \leq b_{n-1}$$ such that $\sum_{i=1}^{n-1} b_i  = 2e-2$. Let $\M (b_{\bullet})$ denote the locally closed locus in $\Mor_e(\PP^1, \PP^n)$ parameterizing nondegenerate, unramified morphisms of degree $e$  such that $$N_f \cong \bigoplus_{i=1}^{n-1} \OO_{\PP^1} (e+b_i).$$  The scheme $\Mor_e(\PP^1, \PP^n)$ is irreducible of dimension $(n+1)(e+1) -1$. The codimension of the locus of vector bundles $E$ on $\PP^1$ with a specified splitting type in the versal deformation space is  given by $h^1(\PP^1, \End(E))$ \cite[Lemma 2.4]{Coskun}. In analogy, we say that the expected codimension of $\M (b_{\bullet})$ is $h^1(\PP^1, \End(N_f))$. Equivalently, the expected dimension is $$(e+1)(n+1)-1- h^1(\PP^1, \End(N_f)).$$ In this paper, we systematically construct examples where $\M (b_{\bullet})$ has many components, some of larger than expected dimension. 

The study of the geometry of $\M(b_{\bullet})$ has a long history.  Celebrated results of Eisenbud, Van de Ven \cite{EisenbudVandeVen}, \cite{EisenbudVandeVen2}, Ghione and Sacchiero \cite{GhioneSacchiero}, \cite{Sacchiero2}, \cite{Sacchiero} characterize the possible splitting types of the normal bundles of rational curves in $\PP^3$ and show that the locus of rational space curves whose normal bundles have a specified splitting type is irreducible of the expected dimension. Similarly, results of Ramella \cite{Ramella}, \cite{Ramella2} show that the locus of nondegenerate rational curves with a specified splitting type for $f^* T_{\PP^n}$ is irreducible of codimension $h^1(\PP^1, \End(f^* T_{\PP^n}))$ for all $n \geq 3$. The behavior of $\M(b_{\bullet})$ for $n \geq 5$ is in stark contrast to these results. 

Recently, Alzati and Re  \cite{AlzatiRe1} showed that the locus of rational curves of degree $11$ in $\PP^8$ whose normal bundles have the splitting type $\OO(13)^3 \oplus \OO(14)^2 \oplus \OO(15)^2$ is reducible. This was the first indication that the geometry of $\M(b_{\bullet})$ is much more complicated for large $n$. This paper grew out of our attempt to generalize their example. We produce examples of reducible $\M(b_{\bullet})$ in $\PP^5$ with $e < 11$, we find $\M(b_{\bullet})$ with arbitrarily many components, and show that the difference between the expected dimension and actual dimension of a component of $\M(b_{\bullet})$ can grow arbitrarily large.

We now summarize our results in greater detail.  First, following Sacchiero \cite{Sacchiero2}, we explain that $\M(b_{\bullet})$ is nonempty provided that $b_1 \geq 2$ and $e\geq n$ (see Theorem  \ref{thm-Sacchiero}). This already shows that the loci $\M(b_{\bullet})$ in general do not have the expected dimension (see Proposition \ref{prop-wrongdim}). 

Before stating the rest of the results, we need some notation.  Let $d$ and $k$ be positive integers and let $n$ be an integer such that $n \geq k+1$. Assume $2e \geq (n-1)d + n-k+1$. Observe, then, that $2e-2 \geq dk$. Let $q$ and $r$ be the quotient and remainder in $$2e-2-dk = q(n-1-k) + r.$$ Let $b_{\bullet}(d^k)$ denote the sequence $$b_1 = \dots = b_k = d, \quad b_{k+1} = \dots = b_{n-r-1} = q, \quad b_{n-r} = \dots = b_{n-1} =q+1.$$ Miret \cite{Miret} has shown that the locus $\M(b_{\bullet}(d))$ is irreducible of the expected dimension. In contrast,  we show the following.

\begin{theorem*}[Theorem \ref{higherDegreeRelations}]
Let  $k\geq 2$ be an even integer. Let $n \geq 3k-1$ and assume that $e$ is sufficiently large. Then $\M(b_{\bullet}(d^k))$ has at least $\frac{k}{2}+1$ irreducible components.
\end{theorem*}

When $d=2$, we  obtain sharper bounds. We classify the components of $\M(b_{\bullet}(2^2))$ in detail. We find that it has two components, one of the expected dimension and the other of larger than expected dimension provided $e$ is sufficiently large (see Theorem \ref{twoConicClassification}). More generally, we study $\M(b_{\bullet}(2^k))$ in greater  detail.

\begin{theorem*}[Theorem \ref{thm-moreconics}]
Let $3k-1 \leq n$, and $e > 2kn-2n-2$.  Then $\M(b_{\bullet}(2^k))$ has at least $k$ components.
\end{theorem*}

As a source of examples, we determine the splitting type of the normal bundle to immersed monomial rational curves  (see Theorem \ref{monomialSplitting}). There has been recent interest in computing these normal bundles (see \cite{AlzatiRe2}). Our methods allow us to compute these normal bundles easily.

\subsection*{Organization of the paper} In \S \ref{sec-prelim}, we collect basic facts concerning normal bundles of rational curves and summarize results of Sacchiero, Ramella and Miret on the stratification of the space of rational curves according to the splitting types of the normal or restricted tangent bundles. In \S \ref{sec-monomial}, we discuss the normal bundles of rational curves defined by monomials. In \S \ref{sec-conics}, we study the spaces $\M(b_{\bullet}(2^k))$ and show that the number of their components grows linearly with $k$ provided $e$ and $n$ are sufficiently large. We also show that if $n\geq 5$ and $e$ is sufficiently large, $\M(b_{\bullet}(2^2))$ has two irreducible components and describe the components. In \S \ref{sec-higherdegree}, we study loci $\M(b_{\bullet}(d^k))$ for $d>2$. Finally, in \S \ref{sec-examples}, we give some examples.

\subsection*{Acknowledgments} We are grateful to A. Alzati, L. Ein, J. Harris, R. Re and J. Starr for invaluable mathematical discussions and correspondence on normal bundles of rational curves. We thank the referee for many helpful suggestions.

\section{Preliminaries}\label{sec-prelim}
In this section, we  recall basic facts concerning the geometry of the space of rational curves in $\PP^n$. We also review Ramella's results on the splitting of $f^* T_{\PP^n}$  \cite{Ramella}, \cite{Ramella2},  Sacchiero's results showing that all possible splittings for the normal bundle occur \cite{Sacchiero2} and Miret's result \cite{Miret} on the irreducibility of $\M(b_{\bullet}(d))$.

\subsection{Basic facts}
Let $E$ be a vector bundle of rank $r$ on $\PP^1$. By Grothendieck's theorem, every vector bundle on $\PP^1$ is a direct sum of line bundles. Hence, there are uniquely determined integers $a_1 \leq a_2  \leq \dots \leq a_r$ such that $E \cong \bigoplus_{i=1}^r \OO(a_i)$. These integers are called the {\em splitting type} of $E$. The vector bundle is called {\em balanced} if  $a_j - a_i \leq 1$ for $1 \leq i < j \leq r$. The dimension of automorphisms of $E$ is given by $h^0(E \otimes E^*)$. Let $V$ be a balanced vector bundle of the same degree and rank as $E$. In particular, $h^0(V \otimes V^*) = \chi(V \otimes V^*)$. Hence, by Lemma \cite[Lemma 2.4]{Coskun}, the {\em expected codimension of a splitting type} is equal to $h^0(E\otimes E^*) - h^0(V \otimes V^*)$ and, by Riemann-Roch, is given by $$h^1(\End(E))= h^1(E^* \otimes E) = \sum_{\{i,j | a_i - a_j\leq -2\}} (a_j - a_i -1).$$ 

A rational curve $C$ of degree $e$ in $\PP^n$ is the image of a morphism $f: \PP^1 \rightarrow \PP^n$, where $$f= (f_0 : \dots : f_n)$$ is defined by homogeneous polynomials $f_i(s,t)$ of degree $e$ without common factors. We always assume that $f$ is {\em birational onto its image} and that the image is {\em nondegenerate}.  We say that the curve $C$ is {\em immersed} or the morphism $f$ is {\em unramified} if the natural map $f^* \Omega_{\PP^n} \to \Omega_{\PP^1}$ is surjective.\footnote{In the literature, authors describe the same condition commonly as $C$ has {\em ordinary singularities} (see \cite{EisenbudVandeVen} and \cite{Sacchiero}). In this paper, we avoid this terminology.} In this case, the kernel is identified with the conormal bundle $N_{f}^*= \mathcal{H}om(N_{f}, \OO_{\PP^1})$, where $N_{f}$ is the normal sheaf. We conclude that $N_{f}$ is a vector bundle of rank $n-1$ and degree $e(n+1)-2$. 

Let $$\partial f = \left( \begin{array}{ccc}  \partial_s f_0 & \dots & \partial_s f_n \\ \partial_t f_0 & \dots & \partial_t f_n \end{array} \right)$$ denote the transpose of the Jacobian matrix. For an unramified morphism, the Euler sequences for $\Omega_{\PP^n}$ and $\Omega_{\PP^1}$ 
 induce a surjective map $$\OO_{\PP^1}(-e)^{n+1}  \stackrel{\partial f}{\longrightarrow} \OO_{\PP^1}(-1)^2$$ and identify the conormal bundle $N_f^{*}$ with the kernel of $\partial f$ \cite{GhioneSacchiero} \cite{Sacchiero2}.  Thus, the normal bundle $N_f$ has splitting type $ \bigoplus_{i=1}^{n-1} \OO(e+b_i)$ if and only if  the kernel of the map $\partial f$ has  splitting type $\bigoplus_{i=1}^{n-1} \OO(-e-b_i)$. In other words, the space of relations among the columns of $\partial f$ is generated by forms of degree $b_i$ for $1 \leq i \leq n-1$. We may view a relation of degree $b_i$ among the columns of $\partial f$ as a parameterized rational curve of degree $b_i$ in $\PP^{n*}$. We will frequently discuss the geometry of the rational curves defined by these relations.

We will need to use the following basic observation. 
\begin{lemma}\label{lem-basicpolynomial}
Let $(f_0(s,t), \dots, f_n(s,t))$ be  an $(n+1)$-tuple of homogeneous polynomials of degree $e$ in $s,t$. Let $(a_0, \dots, a_n)$ be an $(n+1)$-tuple of homogeneous polynomials of degree $b$ in $s,t$.
If $\sum_{i=0}^n a_i \partial_s f_i = \sum_{i=0}^n a_i \partial_t f_i =0$, then $\sum_{i=0}^n f_i \partial_s a_i = \sum_{i=0}^n f_i \partial_t a_i =0.$
\end{lemma}
\begin{proof}
By Euler's relation, the equalities \[ \sum_{i=0}^n a_i \partial_j f_i = 0, \quad j \in \{s,t\} \]
imply 
\[ \sum_{i=0}^n a_i f_i = 0. \]
Differentiating this relation, we see
\[ \sum_i f_i \partial_j a_i + \sum_i a_i \partial_j f_i = \sum_i f_i \partial_j a_i = 0 . \]
\end{proof}

\begin{corollary}
We can write $$N_{f} = \bigoplus_{i=1}^{n-1} \OO(e+b_i),$$ where $2 \leq b_1 \leq \dots \leq b_{n-1}$ and  $\sum_{i=1}^{n-1} b_i = 2e-2$.  
\end{corollary}
\begin{proof}
To see that $b_1 \geq 2$, we can argue as follows. If $b_1=1$,  the map $\OO(-e-1) \to \OO(-e)^{n+1}$ gives a linear relation among the partial derivatives of $f_i$. By Lemma \ref{lem-basicpolynomial}, we obtain a scalar relation among the $f_i$. Hence, the map $f$ is degenerate, contrary to assumption.
\end{proof}

\subsection{The splitting type of the restricted tangent bundle}

The Euler sequence
\[ 0 \to f^* \Omega_{\PP^n} \to \OO(-e)^{n+1} \to \OO \to 0  \]
identifies $f^* \Omega_{\PP^n}$ as the kernel of the homomorphism induced by $f$. Consider the family of homomorphisms $\Hom(\OO(-e)^{n+1}, \OO)$. The Kodaira-Spencer map $$\kappa: \Hom(\OO(-e)^{n+1}, \OO) \to \Ext^1(f^* \Omega_{\PP^n}, f^* \Omega_{\PP^n})$$ factors through the natural morphisms  $$\Hom(\OO(-e)^{n+1}, \OO) \stackrel{\phi}{\longrightarrow} \Ext^1(\OO(-e)^{n+1}, f^* \Omega_{\PP^n}) \stackrel{\psi}{\longrightarrow} \Ext^1( f^* \Omega_{\PP^n},f^* \Omega_{\PP^n}),$$ where $\phi$ and $\psi$ are maps in the long exact sequence obtained by applying $\Hom(\OO(-e)^{n+1}, -)$ and $\Hom(-, f^*\Omega_{\PP^n})$, respectively. Since $\Ext^1(\OO(-e)^{n+1}, \OO(-e)^{n+1}) = 0$ and $\Ext^2(\OO, f^*\Omega_{\PP^n})=0$, we conclude that both $\phi$ and $\psi$ are surjective. Therefore, the Kodaira-Spencer map is surjective for unramified $f$. In fact,  Ramella more generally proves the following.

\begin{theorem}\cite{Ramella}
The locally closed locus in $\Mor_e(\PP^1, \PP^n)$ parameterizing unramified morphisms where $f^* T_{\PP^n}$ has a specified splitting type is irreducible of the expected dimension $(e+1)(n+1) -1 - h^1(\End(f^* T_{\PP^n}))$.
\end{theorem}
Alzati and Re \cite{AlzatiRe3} have further studied the geometry of the loci of rational curves where $f^* T_{\PP^n}$ has a specified splitting type.

\subsection{The possible splitting types of the normal bundle}  In this subsection, we recall Sacchiero's construction of an unramified  $f$ with a specified splitting type for its normal bundle (see \cite{Sacchiero2}). We will use this construction throughout the paper. For our purposes, the relations among the columns of $\partial f$ will be especially important.

Let $\delta_1 = 1$, $\delta_i = b_{i-1} - \delta_{i-1}$ for $1 < i \leq n-1$.  Let $c = 1 + \sum_{i=1}^{n-1} \delta_i$.  Let $p(s,t)$ and $q(s,t)$ be general polynomials of degree $e-c$ (it is enough to assume that $p$ and $q$ do not have common roots or multiple roots and are not divisible by $s$ or $t$).  Let $k_i = c - \sum_{j=1}^{i} \delta_j$ for $0 \leq i \leq n-1$. Observe that $k_0=c,  k_1=c-1$ and $k_{n-1} =1$. Let $f$ be given by the tuple
\[ f = (s^{k_0} p, \ s^{k_1}t^{c-k_1} p, \ s^{k_2}t^{c-k_2} p, \dots, s^{k_{n-2}}t^{c- k_{n-2}} p, \ s^{k_{n-1}} t^{c-k_{n-1}} p, \ t^c q) .\]

\begin{lemma}[Sacchiero's Lemma \cite{Sacchiero2}]\label{lem-Sacchiero}
The map $f$ is unramified and $N_f \cong \bigoplus_{i=1}^{n-1} \OO(e+b_i)$. 
\end{lemma}

\begin{proof} We briefly sketch aspects of Sacchiero's argument that we will later invoke.
A simple calculation shows that $f$ is unramified. Computing $N_f$ is equivalent to computing the kernel of the map $$\partial f: \OO(-e)^{n+1} \rightarrow \OO(-1)^2.$$ We first describe $n-2$ relations satisfied by the columns of $\partial f$. Let $R_i$ for $1 \leq i \leq n-2$ be the row vector $(a_0, \dots, a_n)$, where $a_j =0$ for $j \not= i-1,i, i+1$ and $$ a_{i-1} = (k_{i} - k_{i+1})t^{k_{i-1} - k_{i+1}}, \quad a_i = - (k_{i-1}-k_{i+1})s^{k_{i-1}-k_i}t^{k_i-k_{i+1}}, \quad a_{i+1} = (k_{i-1} - k_i)s^{k_{i-1}-k_{i+1}}.$$ It is easy to see that the columns of $\partial f$ satisfy $R_i$. Let $R$ be the matrix whose rows are $R_i$ for $1 \leq i \leq n-2$. Then $R$ defines a map
$$R: \bigoplus_{i=1}^{n-2} \OO(-e-b_i) \rightarrow \OO(-e)^{n+1}.$$ Since the image of $R$ is contained in the kernel of $\partial f$, the map $R$ factors through the inclusion of $N_f^* \rightarrow \OO(-e)^{n+1}$. An easy computation shows that  $(n-2) \times (n-2)$ minor of $R$ obtained by omitting the first two and the last columns is $\prod_{i=1}^{n-2} (k_{i-1} - k_i) s^{k_{i-1} - k_{i+1}}$. Similarly, the $(n-2) \times (n-2)$ minor of $R$ obtained by omitting the last three columns is  $\prod_{i=1}^{n-2} (k_i - k_{i+1}) t^{k_{i-1} - k_{i+1}}$. Since these minors never simultaneously vanish on $\PP^1$, we conclude that the rank of $R$ is always equal to $n-2$. Hence, the image of $R$ is a subbundle of $N_f^*$. Hence, by degree considerations, we obtain an exact sequence 
$$0 \rightarrow \bigoplus_{i=1}^{n-2} \OO(-e-b_i) \rightarrow N_f^* \rightarrow \OO(-e-b_{n-1}) \rightarrow 0.$$
To conclude that $N_f^*$ is isomorphic to $\bigoplus_{i=1}^{n-1} \OO(-e-b_i)$, it suffices to observe that there are no nontrivial extensions of this form provided $b_{n-1} \geq \max_{1 \leq i \leq n-2} \{b_i\} -1.$
\end{proof}

\begin{remark}\label{rem-orderb}
Note that in the proof we did not need to use that $b_1 \leq b_2 \leq \dots \leq b_{n-1}$. We only needed that $b_{n-1} \geq \max_{1 \leq i \leq n-2} \{b_i\} -1.$ This simplification will make certain constructions later in the paper simpler.
\end{remark}

The next corollary follows from the proof of Lemma \ref{lem-Sacchiero} and Remark \ref{rem-orderb}.

\begin{corollary}
\label{cor-examples}
Let $1=\delta_1, \delta_2, \dots, \delta_{n-1}$ be a sequence of positive integers, $e \geq n$ an integer, and
\[b_{n-1} = 2e-2-2\sum_{i=1}^{n-1} \delta_i + \delta_1 + \delta_{n-1} .\]
Assume 
\[ b_{n-1} \geq \max_i \{ \delta_i + \delta_{i+1} \}-1.\]  
Then there is a nondegenerate rational curve $f$ of degree $e$ in $\PP^n$ with normal bundle $\bigoplus_i \OO(e+b_i)$, where $b_i = \delta_{i}+\delta_{i+1}$ for $1 \leq i \leq n-2$.  The columns of $\partial f$ satisfy the relations $R_i$ from the proof of Lemma \ref{lem-Sacchiero}.
\end{corollary}

Sacchiero uses Lemma \ref{lem-Sacchiero} to deduce the following theorem.

\begin{theorem}\label{thm-Sacchiero}\cite{Sacchiero2}
For $1 \leq i \leq n-1$, let $b_i \geq 2$ satisfy $\sum_{i=1}^{n-1} b_i = 2e-2$.  Then there exists an unramified map $f: \PP^1 \to \PP^n$ such that 
\[ N_{f} = \bigoplus_{i=1}^{n-1} \OO(e+b_i) .\]
In particular, the general smooth rational curve in $\PP^n$ has balanced normal bundle.
\end{theorem}

Other authors (see \cite{Ran}) have studied the generic splitting type of normal bundles of rational curves and described the loci where the splitting is not generic. 

Sacchiero's Theorem implies that unlike the restricted tangent bundle, the stratification of the space of rational curves by the splitting type of the normal bundle is not well-behaved.

\begin{proposition}\label{prop-wrongdim}
For $n \geq 6$, there are nonempty loci $\M(b_{\bullet})$ where the expected codimension is larger than the dimension of  $\Mor_e(\PP^1, \PP^n)$.  In particular, when $(n-2)(2e-2n-1) \geq (e+1)(n+1)$, $M(b_{\bullet}(2^{n-2}))$ is nonempty even though its expected dimension is negative.
\end{proposition}
\begin{proof}
We compute the expected codimension for curves with normal bundle $$N_{f} = \OO(e+2)^{n-2} \oplus \OO(3e-2n+2).$$  The  expected codimension is
\begin{align*}
h^1(\End(\OO(e+2)^{n-2} \oplus \OO(3e-2n+2))) &= (n-2)(3e-2n+2-(e+2)-1)& \\  &= (n-2)(2e-2n-1).& 
\end{align*}

For fixed $n$, this expression grows like $2(n-2)e$ with $e$. On the other hand, the dimension of $\Mor_e(\PP^1, \PP^n)$ grows like $(n+1)e$ with $e$. For $n > 5$, $2(n-2)e$ grows faster than $(n+1)e$. Hence, for sufficiently large $e$ the expected codimension of the locus $\M(b_{\bullet}(2^{n-2}))$ is larger than the dimension of $\Mor_e(\PP^1, \PP^n)$.

\end{proof}

\begin{remark}
As the referee pointed out, since $\Mor_e(\PP^1, \PP^n)$ admits an action of $\PP GL(2)$ that preserves normal bundles, any nonempty locus has dimension at least 3. Hence, one would have expected $M(b_{\bullet}(2^{n-2}))$ to be empty under the weaker inequality  $(n-2)(2e-2n-1) \geq (e+1)(n+1)-3$.
\end{remark}

Finally, Miret showed that if we fix only the lowest degree factor of the normal bundle, then the resulting locus $\M(b_{\bullet}(d))$ is irreducible.  In this case, Eisenbud and Van de Ven's and Ghione and Sacchiero's proofs of irreducibility for $\PP^3$ generalize with little change.
\begin{theorem}
\label{uniqueLowestFactor}\cite{Miret}
Let $2e-2 \geq d(n-1) + n-2$. Then the locus $\M(b_{\bullet}(d))$ is irreducible of the expected dimension.
\end{theorem}

\section{Monomial curves}\label{sec-monomial}
In \S \ref{sec-prelim}, we saw that computing the normal bundle $N_f$ corresponds to determining the kernel of the map $\partial f$. In general, this is a hard linear algebra problem. However,  for monomial maps there is a simple way of reading off the normal bundle from the terms in the sequence.  Since monomial maps provide useful examples, we describe the procedure in detail here.  Our approach appears to be easier than that of Alzati, Re, and Tortora in \cite{AlzatiRe2}.

First, it is easy to decide when monomial maps are unramified.

\begin{lemma}
Let $k_n=0 < k_{n-1} < \cdots < k_1 < k_0=e$ be a sequence of increasing integers.  Let $f = (s^{e}, s^{k_1}t^{e-k_1}, \dots, s^{k_{n-1}} t^{e-k_{n-1}}, t^{e})$ be a monomial map.  Then $f$ is unramified if and only if $k_1=e-1$ and $k_{n-1} = 1$, which implies that the image of $f$ is smooth.
\end{lemma}
\begin{proof}
First, we show that a map with $k_1=e-1$ and $k_{n-1} = 1$ is unramified.  Consider the matrix of partials coming from only considering $f_0, f_1, f_{n-1}$ and $f_n$.  We have
\[ \left[ \begin{array}{llll}
es^{e-1} & (e-1)s^{e-2}t & t^{e-1} & 0 \\
0 & s^{e-1} & (e-1)st^{e-2} & et^{e-1}
\end{array} \right] .\]
We see that if $s \neq 0$, the first two columns are independent, and if $t \neq 0$, the last two columns are independent, so the map is unramified.  Moreover, the image of $f$ is smooth, since the curve $(s^e, s^{e-1}t, st^{e-1}, t^e)$ is smooth and is a projection of the image of $f$.

Now suppose $k_1 \neq e-1$.  We show $f$ is ramified (the case $k_{n-1} \neq 1$ follows by symmetry).  If $k_1 \neq e-1$, then we see that $t$ divides $\partial_t f_i$ for every $i = 0, \dots n$.  Thus, at the point $t=0$, the curve  is ramified, since $\partial_t f$ is identically zero.  This completes the proof.
\end{proof}

\begin{theorem}
\label{monomialSplitting}
Let $k_n=0 < 1=k_{n-1} < \cdots < k_1 = e-1 < k_0=e$ be a sequence of increasing integers.  Let $f = (s^{k_0}, s^{k_1}t^{e-k_1}, \dots, t^{e-k_{n}})$ be an unramified map  whose coordinates are given by monomials.  Then 
\[ N_f \cong \bigoplus_{i=1}^{n-1} \OO(e+b_i) \] 
where 
\[ b_i = k_{i-1}-k_{i+1}. \]
\end{theorem}
\begin{proof}
The proof of this theorem is similar to, and in fact easier than, the proof of Lemma \ref{lem-Sacchiero}.  Computing the normal bundle is equivalent to computing the kernel of the map $$\partial f : \OO(-e)^{n+1} \to \OO(-1)^2.$$  Hence, we would like to find generators for the relations among the columns of $\partial f$. First, we exhibit $n-1$ independent relations among the $(\partial_s f_j, \partial_t f_j)$.  Each relation $R_i = (a_0, \dots, a_n)$ only has three nonzero terms, $a_{i-1}$, $a_i$ and $a_{i+1}$ for $1 \leq i \leq n-1$.   Then $$a_{i-1} = (k_i-k_{i+1})t^{k_{i-1}-k_{i+1}}, \quad a_i = - (k_{i-1}-k_{i+1})s^{k_{i-1}-k_i}t^{k_i-k_{i+1}}, \quad a_{i+1} = (k_{i-1}-k_i)s^{k_{i-1}-k_{i+1}},$$ and $a_j = 0$ for $j \neq i-1, i, i+1$ is a relation, since it is easily checked that
\[ a_{i-1} \left[ \begin{array}{c}  k_{i-1} s^{k_{i-1}-1}t^{e-k_{i-1}} \\ (e-k_{i-1})s^{k_{i-1}} t^{e-k_{i-1}-1} \end{array} \right]  
+ a_i  \left[ \begin{array}{c} k_is^{k_i-1}t^{e-k_i} \\ (e-k_i)s^{k_i}t^{e-k_i-1} \end{array} \right]  + a_{i+1} \left[ \begin{array}{c} k_{i+1}s^{k_{i+1}-1}t^{e-k_{i+1}} \\ (e-k_{i+1})s^{k_{i+1}}t^{e-k_{i+1}-1} \end{array} \right]  = 0
\]

Let $R$ be the matrix whose rows are the relations $R_i$.  Note that $R_i$ consists of a row of homogeneous polynomials of degree $b_i = k_{i-1} - k_{i+1}$ for $1 \leq i \leq n-1$.
Consequently, the matrix $R$ defines a map $$\bigoplus_{i=1}^{n-1} \OO(-e-b_i) \stackrel{R}{\longrightarrow} \OO(-e)^{n+1},$$ whose image is in the kernel of $\partial f$. Therefore, the map $R$ factors through the inclusion
$$0 \longrightarrow N_f^* \longrightarrow \OO(-e)^{n+1}.$$ 
Next, we claim that $R$ has rank $n-1$ at every point of $\PP^1$, hence induces an isomorphism $$\bigoplus_{i=1}^{n-1} \OO(-e-b_i) \cong N_f^*.$$ The theorem easily follows.
The $(n-1) \times (n-1)$ minors of $R$ obtained by omitting the first two columns and the last two columns are easy to compute and are given by
$$\prod_{i=1}^{n-1}  (k_{i-1} - k_{i}) s^{k_{i-1}-k_{i+1}} \quad \mbox{and} \quad \prod_{i=1}^{n-1}(k_i - k_{i+1}) t^{k_{i-1}-k_{i+1}},$$ respectively. Since these minors do not simultaneously vanish, we conclude that $R$ has rank $n-1$ at every point of $\PP^1$. 
\end{proof}

Using the same technique, we can  also compute the restricted tangent bundle of a curve generated by monomial ideals. The Euler sequence 
$$0 \longrightarrow f^* \Omega_{\PP^n} \longrightarrow \OO(-e)^{n+1} \stackrel{f}{\longrightarrow} \OO \longrightarrow 0$$ exhibits $f^* \Omega_{\PP^n}$ as the kernel of the map defined by $f$. The columns of $f$ satisfy the $n$ relations given by $$ t^{k_{i-1}- k_i} f_{i-1} - s^{k_{i-1}-k_i} f_i=0.$$ The argument in the proof of Theorem \ref{monomialSplitting} allows us to conclude the following proposition. 

\begin{proposition}
Let $k_n=0 < k_{n-1} < \cdots < k_1 < k_0=e$ be a sequence of increasing integers. Let $f = (s^{k_0 }, s^{k_1}t^{e-k_1}, \dots, t^{e-k_{n} })$ be an unramified monomial map.  Then  
\[  f^* T_{\PP^n} \cong \bigoplus_{i=1}^{n} \OO(e+c_i), \] 
where 
\[ c_i = k_{i-1}-k_{i}. \]
\end{proposition}

We conclude this section with a short discussion of ramified monomial maps.  If $f$ is ramified, then $N_{f}$ has both a torsion part and a free part.  Taking duals and repeating the argument from the smooth case, we see that $N_{f}^{*}$ is the kernel of the map $\OO(-e)^{n+1} \to \OO(-1)^2$ given by the partials of $f$, only now the map has a cokernel corresponding to the torsion sheaf $\Ext^1(N_{f}, \OO)$.  Our calculation in Theorem \ref{monomialSplitting} still works in this case for computing the splitting type of $N_{f}^{*}$.

\section{Dimensions of Components}\label{sec-conics}

In this section we prove many of the main results of the paper. We will introduce a natural incidence correspondence $\aA$ that parameterizes maps $f$ together with associated syzygy relations that determine $N_f$. There are natural projection maps $\pi_1$ and $\pi_2$ from $\aA$ to  the space of syzygies and to $\M(b_{\bullet})$. We describe the fibers of $\pi_1$ for certain syzygies and compute the dimensions of these fibers. This allows us to exhibit different components of the incidence correspondence $\aA$. For each component of $\aA$, we will then exhibit an example where the fiber dimension of $\pi_2$ is equal to an a priori lower bound. This will show that $\pi_2$ is relatively flat of relative dimension equal to the a priori lower bound and will compute the dimension of this component of $\M(b_{\bullet})$.  

We start by working out the expected dimension of $\M(b_{\bullet}(d^k))$ and proving an a priori lower bound on the dimension.

\begin{lemma}
\label{expCodimLemma}
Assume that $2e \geq (d+1) (n-1) -k +2$. Then the expected codimension of $\M(b_{\bullet}(d^k))$ in $\Mor_e(\PP^1, \PP^n)$ is $k(2e+1+k) - (d+1)nk$. This is an upper bound for the codimension of every component of  $\M(b_{\bullet}(d^k))$.
\end{lemma}
\begin{proof}
The expected codimension is by definition $h^1(\End(N))$. Recall that $q$ and $r$ are defined by the expression
$$2e-2-dk = (n-k-1)q+r .$$ 
Since $$N \cong \OO(e+d)^{k} \oplus \OO(e+q+1)^r \oplus \OO(e+q)^{n-1-k-r},$$ we see that 
\begin{align*}
h^1(\End(N)) &= h^1\left(\OO (d-q-1)^{kr} \oplus \OO (d-q)^{k(n-1-k-r)}\right)& \\ &= kr(q-d) + k(n-1-k-r)(q-d-1).&
\end{align*}
Simplifying using the fact $2e-2-dk= q(n-1-k)+r$ yields the desired formula. The last statement follows from the fact that the loci  $\M(b_{\bullet}(d^k))$  are determinantal loci. 
\end{proof}

\subsection{Two Conics}
We now classify the components of $\M(b_{\bullet}(2^2))$.  The following definition will be central to our discussion. 

\begin{definition}
Let $\alpha$ and $\beta$ be $(n+1)$ tuples of polynomials, where the $\alpha_i$ all have the same degree and the $\beta_i$ all have the same degree. We refer to $\alpha$ and $\beta$ as \emph{unscaled parameterized curves}. Then $\alpha$ and $\beta$ satisfy the {\em parameterized tangency condition} if for some choice of parameters $s$ and $t$ on $\PP^1$, $\partial_s \alpha$ is a fixed polynomial multiple of $\partial_t \beta$. 
\end{definition}
The tuples $\alpha$ and $\beta$ give rise to a map from $\PP^1$ to $\PP^n$. If $\alpha$ and $\beta$ have the same degree and satisfy the parameterized tangency condition, then $\partial_s \alpha$ is a scalar multiple of $\partial_t \beta$.  If $\alpha$, $\beta$ give rise to nondegenerate conics that satisfy the parameterized tangency condition, then their planes intersect in at least a line. Furthermore, if they intersect in a line $\ell$, then both $\alpha$ and $\beta$ are tangent to $\ell$.

Let $\mathcal{G}$ denote the closure of the locus in $\M(b_{\bullet}(2^2))$ where $\partial f$ has  two independent degree two relations whose corresponding curves in $\PP^{n *}$ lie in disjoint planes. Let $\mathcal{PT}$ denote the closure of the locus in $\M(b_{\bullet}(2^2))$ where $\partial f$ has two independent degree two relations satisfying the parameterized tangency condition.

\begin{theorem}
\label{twoConicClassification}
For $n \geq 5$, $e \geq 2n-3$, $\M(b_{\bullet}(2^2))$ has precisely two components, $\mathcal{G}$ and $\mathcal{PT}$.  The dimension of $\mathcal{G}$ is the expected dimension $e(n-3)+7n-6,$ and the dimension of $\mathcal{PT}$ is $e(n-2)+5n-3$.
\end{theorem}

The proof of the theorem involves a detailed case-by-case analysis of the types of conic relations that can occur among the columns of $\partial f$.  We start by showing that if $f$ is nondegenerate and unramified, then the relations cannot define degenerate conics. 

\begin{lemma}\label{lem-line}
If  $\partial f$ satisfies a degree two relation that defines a two-to-one map to a line, then $f$ is degenerate.
\end{lemma}
\begin{proof}
Let $a$ define the degree two relation.  Then, up to changing coordinates, we can view $a$ as $(s^2, t^2, 0, \dots, 0)$.  By Lemma \ref{lem-basicpolynomial}, we have the relation
\[ \sum_{i=0}^n f_i \partial_s a_i = 0.\]
Hence, $f_0= 0$ and $f$ is degenerate.
\end{proof}

\begin{lemma}\label{lem-reducible}
If $\partial f$ satisfies a degree two relation $a=(a_0, \dots, a_n)$ with all the $a_i$'s having a common root, then $f$ is degenerate.
\end{lemma}
\begin{proof}
Change coordinates on $\PP^1$ so that the common root is given by  $s=0$, and let $a_i = s a'_i$.  Then 
\[ \sum_{i=0}^n a'_i \partial_j f_i = 0 \]
for linear functions $a_i'$. By Lemma \ref{lem-basicpolynomial}, $f$ must be degenerate.
\end{proof}

\begin{corollary}
If $\partial f$ satisfies two degree two relations that define conics in the same plane in $\PP^{n *}$, then $f$ is degenerate.
\end{corollary}
\begin{proof}
Any one dimensional family of degree two maps from $\PP^1$ to $\PP^2$ necessarily contains a degenerate conic. By the previous two lemmas, $f$ is degenerate.
\end{proof}

Thus, to study $M(b_{\bullet}(2^2))$, we need only consider $f$ with $\partial f$ satisfying two relations that define smooth conics not lying in the same plane.  Hence, there are three possibilities: the planes spanned by the conics could be disjoint, meet in a point, or meet in a line.  First, we study the case when the degree two relations on $\partial f$ define conics with disjoint planes. 

\begin{proposition}
\label{goodComp}
If $3k \leq n+1$ and $2e \geq 3(n-1)$, there is a component of $\M(b_{\bullet}(2^k))$ of the expected dimension such that for the general element $f$, the degree two relation on  $\partial f$ define $k$ general conics in $\PP^{n*}$.
\end{proposition}

\begin{definition}
Let $\mathcal{C}^k$ be the space of linearly independent ordered $k$-tuples $(a_1, \dots, a_k)$, where each $a_i$ is an unscaled parameterized conic. Let $\aA$ be the incidence correspondence of tuples $(a_1, \dots, a_k, f)$ such that $(a_1, \dots, a_k) \in \mathcal{C}^k$, $f \in \M(b_{\bullet}(2^k))$, and $\sum_{j=0}^n a_{ij} \partial_l f_j =0$ for $1 \leq i \leq k$ and $l \in \{ s,t \}$. Let $\overline{\aA}$ denote the closure of $\aA$ in $\mathcal{C}^k \times\Mor_e(\PP^1, \PP^n)$. 
\end{definition}

Note that $f$ is defined only up to scaling, while the $a_i$ are tuples of polynomials. This will play a role in the dimension counts later. The incidence correspondence $\overline{\aA}$ projects  via $\pi_1$ to the space $\mathcal{C}^k$ and via $\pi_2$ to $\Mor_e(\PP^1, \PP^n)$. We will estimate the dimensions of the fibers of these two projections.

\begin{proof}[Proof of Proposition \ref{goodComp}]
 We show that there is one component $\Gamma$ of the incidence correspondence $\overline{\aA}$ that dominates $\mathcal{C}^k$, and the general element of $\Gamma$ maps to $\M(b_{\bullet}(2^k))$.

First, we compute the dimension of $\mathcal{C}^k$.  An unscaled parameterized conic in $\PP^{n*}$ is determined by specifying the plane it spans and a degree two map into that plane. The dimension of the Grassmannian $\mathbb{G}(2,n)$ is $3(n-2)$ and the parameterized map is given by specifying the 9 coefficients of  the three polynomials of degree $2$.  We conclude that $\mathcal{C}^k$ has dimension $3(n-2)k+9k$.  

We claim the general fiber of $\pi_1$ has dimension $(e+1)(n+1)-2k(e+2)-1$.  Let $\mathcal{C}^{\circ}$ denote the Zariski open locus in $\mathcal{C}^k$ parameterizing $k$-tuples of conics that span linearly independent planes, and let $(a_1, \dots, a_k) \in \mathcal{C}^{\circ}$.  Choose coordinates so that $a_i$ is the conic $x_{3i-3} = s^2, x_{3i-2} = -2st, x_{3i-1} = t^2$ in the linear space $\{x_0 = \dots = x_{3i-4} = 0 = x_{3i} = \dots = x_n \}$.  The conic $a_i$  imposes conditions only on $f_{3i-3}, f_{3i-2}$ and $f_{3i-1}$. Hence, the number of conditions imposed by the $k$ conics is $k$ times the number of conditions imposed by one conic.    By Lemma \ref{lem-basicpolynomial} the conditions
\[ \sum_{j=0}^n a_{ij} \partial_l f_j = 0, \: \: l \in \{s,t\} \quad \mbox{translate to} \quad \sum_{j=0}^n f_j \partial_l a_{ij} = 0, \: \: l \in \{s,t\}.  \]
Hence,  $$2s f_{3i-3}-2t f_{3i-2} = 0 \quad \mbox{and} \quad -2sf_{3i-2} + 2t f_{3i-1} = 0.$$  This shows that $st | f_{3i-2}$, but that $\frac{f_{3i-2}}{st}$ can be any degree $e-2$ polynomial, and that $\frac{f_{3i-2}}{st}$ completely determines $f_{3i-3}, f_{3i-2}$, and $f_{3i-1}$.  Therefore, each conic imposes $2(e+2)$ conditions, and the general fiber of $\pi_1$ has dimension $(e+1)(n+1)-2k(e+2)-1$. Notice that this dimension is positive under our assumption that $3k \leq n+1$.

Since the fibers of $\pi_1$ over $\mathcal{C}^{\circ}$ are irreducible of constant dimension, $\pi_1^{-1}(\mathcal{C}^{\circ})$ is  irreducible. Let $\Gamma$ be the closure of  $\pi_1^{-1}(\mathcal{C}^{\circ})$ in $\overline{\aA}$. Then $\Gamma$ is irreducible, dominates $\mathcal{C}^k$ and 
 $$\dim(\Gamma) = (e+1)(n+1) + 3kn - 2ek -k-1.$$  
 
We now compute the dimension of the general fiber of the map $\pi_2|_{\Gamma}$.  In the next paragraph, we construct an example of a parameterized curve $f \in \pi_2(\Gamma) \cap \M(b_{\bullet}(2^k))$. It follows that $\pi_2$ maps the general element of $\Gamma$ into $\M(b_{\bullet}(2^k))$. The general fiber of $\pi_2$ over $\pi_2(\Gamma) \cap \M(b_{\bullet}(2^k))$ corresponds to a choice of $k$ unscaled parameterized conics spanning the $k$-dimensional vector space of conic relations on $\partial f$. Hence, this fiber has dimension $k^2$. We conclude that $\M(b_{\bullet}(2^k))$ has a component of dimension  $$(e+1)(n+1) -k(2e+k+1)+3nk -1.$$  This matches the expected dimension by Lemma \ref{expCodimLemma}.

To finish, it suffices to construct an example $f$ where $\partial f$ satisfies $k$ general conic relations.  Using the division algorithm, write $2e-2 - 2k = q(n-k-1) + r$ with $0 \leq r < n-k-1$. We construct a curve with $$N_{f} = \OO(e+2)^k \oplus \OO(e+q)^{n-k-1-r} \oplus \OO(e+q+1)^r.$$  The construction depends on whether $n-k$ is odd or even.  In the odd case, we can construct a monomial example.
\begin{itemize}
\item If $n-k$ is odd, then $n-k-1$ and $r$ are even, thus $\frac{n-k-1}{2}$ and $\frac{r}{2}$ are integers. Now define the sequence
\begin{equation}\label{seq1}
1,1,x_1, 1,1,x_2, \dots, 1, 1, x_k, 1, x_{k+1}, 1, x_{k+2} , \dots, x_{\frac{n-k-1}{2}}, 1,
\end{equation}
where $x_1 = \dots =x_{\frac{r}{2}} = q$ and $x_{\frac{r}{2}+1} = \dots =  x_{\frac{n-k-1}{2}} = q-1$.  Set $k_0=e$ and for $1 \leq i \leq n$ define a sequence $k_i$ by the property that $k_i - k_{i-1}$ is equal to the $i$th entry of  Sequence (\ref{seq1}).  Note that $k_{n-1}=1$ and $k_n =0$. Let $f$ be the unramified monomial map 
\[ f = (s^{k_0}, s^{k_1}t^{e-k_1}, s^{k_2}t^{e-k_2}, \dots, s^{k_{n-1}} t^{e-k_{n-1}}, t^{e-k_n}). \]
Write $N_f \cong \bigoplus_{i=1}^{n-1} \OO(e+b_i)$.  By Theorem \ref{monomialSplitting},  $$b_i= k_{i-1} - k_{i+1} = (k_{i-1} - k_i) + (k_i - k_{i+1}),$$ which is the sum of the $i$th and $(i+1)$st entries in  Sequence (\ref{seq1}). Hence, $k$ of the  $b_i$ are equal to $2$, $r$ of the $b_i$ are equal to $q+1$ and the rest are equal to $q$.  Therefore, the normal bundle has the required form. Moreover, by the proof of Theorem \ref{monomialSplitting}, $\partial f$ satisfies $k$ degree two relations that define conics in $\PP^{n*}$ with independent planes. 

\item If $n-k$ is even, then define the sequence
\begin{equation}\label{seq2}
 1, 1, x_1, 1, 1, x_2, \dots, 1, 1, x_k, 1, x_{k+1}, 1, x_{k+2}, \dots, x_{\frac{n-k}{2}-1}, 1,
\end{equation} 
where $x_1 = \dots =x_{\lfloor \frac{r}{2} \rfloor} = q$ and $x_{\lfloor \frac{r}{2} \rfloor+1 } = \dots =  x_{\frac{n-k}{2}-1} = q-1$.  Notice the length of Sequence (\ref{seq2}) is $n-1$. We now invoke Corollary \ref{cor-examples} with $\delta_i$ given by Sequence (\ref{seq2}). Then $b_i= \delta_i + \delta_{i+1}$ for $1 \leq i \leq n-2$. Note that for $i \leq n-2$,  $k$ of the $b_i$ are equal to $2$,  $2 \lfloor \frac{r}{2} \rfloor$ are equal to $q+1$, and  the remaining $2 \lfloor \frac{n-k-r-1}{2} \rfloor$ are $q$. Since $n-k$ is even, either $r$ or $n-k-1-r$ is even. If $r$ is even, then $b_{n-1}$ is $q$. Otherwise, $b_{n-1}$ is $q+1$.  Hence, the hypotheses of Corollary \ref{cor-examples} are satisfied and we obtain an unramified morphism with the required normal bundle. Furthermore, $\partial f$ satisfies the desired relations.

\end{itemize}

\end{proof}

Now we consider the case when the degree two relations on $\partial f$ define conic curves in $\PP^{n*}$ whose planes intersect in a single point.  Our eventual goal is Corollary \ref{twoConicsCor}, which shows that these maps do not give a new component of $\M(b_{\bullet}(2^2))$. Let $\mathcal{P}$ be the locus in $\mathcal{C}^2$ parameterizing ordered pairs of conics whose planes intersect in a single point. 

\begin{lemma}\label{lem-conicsonepoint}
Let $(c_1, c_2) \in \mathcal{P}$. Then the dimension of $\pi_1^{-1} (c_1, c_2)$ is  at most  $(e+1)(n+1) - 4e -7$ and the locus where equality occurs has codimension at least $2$ in $\mathcal{P}$.
Furthermore,  if $(c_1, c_2)$ is general, then $\pi_1^{-1} (c_1, c_2)$ has dimension $(e+1)(n+1) - 4e -9$. 
\end{lemma}
\begin{proof}
By Lemmas \ref{lem-line} and \ref{lem-reducible}, we may assume that $c_1$ and $c_2$ parameterize smooth and nondegenerate conics. In suitable coordinates, we may write them as 
\[ \begin{array}{llllllll}
(g_1, &g_2, &g_3, &0, &0, &0, &\dots, &0) \\
(0, &0, &g_4, &g_5, &g_6, &0, &\dots, &0)
\end{array} \]
where the planes of the conics intersect at the point $(0,0,1,0, \dots, 0)$. Let $M_{ij}$ denote the matrix 
\[ M_{i,j} = \left[ 
\begin{array}{ll}
\partial_s g_i & \partial_s g_j \\
\partial_t g_i & \partial_t g_j
\end{array}
\right] \]
Then we claim $\det M_{i,j}$ is not identically zero for $1 \leq i < j \leq 3$.  
Write $g_i = a_i s^2 + b_i st + c_i t^2$, and notice that
\[ \det M_{i,j} = 2(a_ib_j - a_j b_i)s^2 + 2(b_ic_j-b_jc_i)t^2 + 4(a_ic_j-a_jc_i)st .\]
If the determinant were $0$, then the $2$ by $2$ minors of
\[ \left[
\begin{array}{lll}
a_i & b_i & c_i \\
a_j & b_j & c_j
\end{array}
\right] \]
would vanish, which shows that $g_i$ and $g_j$ are linearly dependent. This forces the first conic to be degenerate contrary to assumption.

Then we see that for any element $(f_0, \dots, f_n)$ in $\pi_1^{-1}(c_1,c_2)$, that
\[ \left[ \: M_{1,2} \: \begin{array}{l}
\partial_s g_3 \\ \partial_t g_3
\end{array} \right] \left[ \begin{array}{l}
f_0 \\ f_1 \\ f_2
\end{array} \right] = 0 .\]

Multiplying by $M_{1,2}^{-1}$ and solving for $f_0$ and $f_1$, we get
\[ \left[ \begin{array}{l}
f_0 \\ f_1
\end{array} \right] = - M_{1,2}^{-1} \left[ \begin{array}{l}
\partial_s g_3 \\ \partial_t g_3
\end{array} \right]f_2 \]
Expanding out $M_{1,2}^{-1}$ in terms of the partials of the $g_i$, we finally see that
$$ f_0 = - \frac{\det M_{2,3}}{\det M_{1,2}} f_2 \quad \mbox{and} \quad  f_1 = - \frac{\det M_{1,3} }{\det M_{1,2}} f_2 .$$
Unless there is cancellation, we see that $\det M_{1,2}$ must divide $f_2$.  If $\det M_{1,2}$ does not divide $f_2$, we see that the $2$ by $2$ minors of
\[ \left[ \begin{array}{lll}
\partial_s g_1 & \partial_s g_2 & \partial_s g_3 \\
\partial_t g_1 & \partial_t g_2 & \partial_t g_3
\end{array} \right] \]
must all vanish at the same point, which means that the conic must be degenerate.

We can obtain similar expressions for $f_3$ and $f_4$ in terms of $f_2$, and similarly can see that $f_2$ must be divisible by $\det M_{5,6}$.  Thus, there are between $e+1-2 = e-1$ and $e+1-4=e-3$ choices for $f_2$ (depending on what factors, if any, $\det M_{1,2}$ and $\det M_{5,6}$ have in common).  Given $f_2$, we see that $f_0, f_1, f_3$ and $f_4$ are completely determined, so the fiber dimension is between $(e+1)(n+1) - 4e-9$ and $(e+1)(n+1) - 4e-7$. Furthermore, the dimension is  $(e+1)(n+1) - 4e-9$ if $\det M_{1,2}$ and $\det M_{5,6}$ have no common factors. The dimension is  $(e+1)(n+1) - 4e-8$ if $\det M_{1,2}$ and $\det M_{5,6}$ have one common factor. Finally, the dimension is  $(e+1)(n+1) - 4e-7$ if $\det M_{1,2}$ and $\det M_{5,6}$ are constant multiples of each other.  Since $\det M_{1,2}$ and $\det M_{5,6}$ are arbitrary degree 2 polynomials, the locus where they have 2 common factors has codimension 2 in $\mathcal{P}$. This concludes the proof.
\end{proof}

\begin{corollary}
\label{twoConicsCor}
Let $n \geq 5$ and $2e \geq 3(n-1)$. Then $\pi_2(\pi_1^{-1}(\mathcal{P})) \cap \M(b_{\bullet}(2^2))$ is contained in the component $\pi_2(\Gamma)$.  
\end{corollary}
\begin{proof}
By Lemma \ref{lem-conicsonepoint}, every component of $\pi_1^{-1}(\mathcal{P}) \subset \aA$ has dimension at most $$\dim(\mathcal{P}) + (e+1)(n+1) - 4e -9.$$ Observe that the fiber dimension of $\pi_2$ over a point in $\pi_2(\pi_1^{-1}(\mathcal{P}))$ is at least $4$. Hence, the image of any such component is at most $\dim(\mathcal{P}) + (e+1)(n+1) - 4e -13$. On the other hand, $\dim(\mathcal{P}) = \dim(\mathcal{C}^2) - n+4$. By Lemma \ref{expCodimLemma}, the minimum possible dimension of a component of $\M(b_{\bullet}(2^2))$ is $\dim(\mathcal{C}^2)+ (e+1)(n+1) - 4e -13$.  Since $n \geq 5$, we conclude that $\pi_2(\pi_1^{-1}(\mathcal{P}))$ cannot contain any irreducible components of $\M(b_{\bullet}(2^2))$. Hence, $\pi_2(\pi_1^{-1}(\mathcal{P})) \cap \M(b_{\bullet}(2^2))$ is contained in $\pi_2(\Gamma)$.
\end{proof}

Let $\mathcal{L}$ denote the locus in $\mathcal{C}^2$ where the planes of the two conics intersect in a line.
\begin{lemma}
\label{paramTangencyLemma}
Let $(q_1, q_2) \in \mathcal{L}$. Then either $q_1$ and $q_2$ satisfy the parameterized tangency condition and the fiber $\pi_1^{-1}(q_1, q_2)$ has dimension $(e+1)(n+1) - 3e-7$ or $\pi_1^{-1}(q_1, q_2)$ contains no points of $\aA$.
\end{lemma}
\begin{proof}
By Lemmas \ref{lem-line} and \ref{lem-reducible}, we assume that the conics are smooth. We choose coordinates on $\PP^{n *}$ so that the two planes have the form
\begin{align*}
(*, *, *, 0, 0, \dots, 0) \\
(0, *, *, *, 0, \dots, 0)
\end{align*}
First, we show that each conic must be tangent to the line of intersection $\ell$.  To get a contradiction, suppose the first conic intersects $\ell$ in two distinct points.  Up to reparameterizing $\PP^1$, we can express the conic as
\[ (st, s^2, t^2, 0, \dots, 0) . \]
Let the other conic be
\[ (0, g_1, g_2, g_3, 0 \dots, 0) , \]
where $g_i = a_i s^2+ b_i st + c_i t^2$.  Then we see that points in $\pi_1^{-1}(q_1,q_2)$ are maps $(f_0, f_1, f_2, f_3,  \dots, f_n)$ with
\[ \left[ \begin{array}{llll}
t & 2s & 0 & 0 \\
s & 0 & 2t & 0 \\
0 & 2a_1s + b_1t & 2a_2s + b_2t & 2a_3 s + b_3t \\
0 & b_1s+2c_1t & b_2s+2c_2t & b_3s+2c_3t 
\end{array} \right]
\left[ \begin{array}{l}
f_0 \\
f_1 \\
f_2 \\
f_3
\end{array}
\right]
=
\left[ \begin{array}{l}
0 \\
0 \\
0 \\
0
\end{array}
\right]
\]
If the $f_i$ are not all zero, we see that this implies that the determinant of the $4$ by $4$ matrix must be $0$, giving us the relation
\[ (a_3b_2 - a_2b_3)s^4 + 2(a_3c_2 - a_2c_3)s^3t + (a_3b_1 - a_1b_3 + b_3c_2 - b_2c_3)s^2t^2 + 2(a_3c_1 - a_1c_3)st^3 + (b_3c_1 - b_1c_3)t^4 = 0 \]
or
\begin{align} \label{l1}
a_3b_2 - a_2b_3 = 0\\ \label{l2}
a_3c_2 - a_2c_3 = 0\\ \label{l3}
a_3b_1 - a_1b_3 + b_3c_2 - b_2c_3 = 0 \\ \label{l4}
a_3c_1 - a_1c_3 = 0\\ \label{l5}
b_3c_1 - b_1c_3 = 0
\end{align}

We consider the implications of this on the matrix with rows given by the coefficients of the $g_i$:
\[ G = \left[ \begin{array}{lll}
a_1 & b_1 & c_1 \\
a_2 & b_2 & c_2 \\
a_3 & b_3 & c_3
\end{array} \right] \]

Equations $(\ref{l1})$ and $(\ref{l2})$ are precisely the vanishing of two of the subminors of the bottom two rows of the matrix $G$. We claim that $a_2$ and $a_3$ are not both $0$. If $a_2 = a_3 = 0$, then by equation $(\ref{l4})$ $a_1c_3 = 0$.  Since the $a_i$ cannot all vanish (otherwise, the $g_i$ would all have a common factor), we see that $c_3 = 0$.  By equation $(\ref{l5})$, this shows that $b_3c_1 = 0$.  Since $b_3 \neq 0$ (otherwise we would have $g_3 = 0$ and one of our conics would be degenerate), this gives that $c_1 = 0$.  Thus, our two conics have the form 
\begin{align*}
&(st, s^2, t^2, 0, 0, \dots, 0) \\
&(0, a_1 s^2 + b_1 st, b_2st+a_1t^2, b_3 st, 0, \dots, 0)
\end{align*} 

We see that $a_1$ times the first row minus the second row consists of a conic whose terms all have a common factor, which is impossible.  Therefore, our original assumption that $a_2=a_3=0$ was wrong.

If $a_2$ and $a_3$ are not both $0$, then equations $(\ref{l1})$ and $(\ref{l2})$ give $b_2 = \lambda a_2$,  $c_2 = \nu a_2$, $b_3 = \lambda a_3$, $c_3 = \nu a_3$.  If $a_3 = 0$, then we see that both $b_3$ and $c_3$ are $0$, which means that $g_3 = 0$, which means that the second conic is a double cover of a line, which is impossible, so we see that $a_3 \neq 0$.  Combining our expressions for $c_3$ with equation $(\ref{l4})$, we see that $a_3(c_1-\nu a_1) = 0$, which means $c_1 = \nu a_1$.  Thus, the determinant of $G$ is $0$ since the last column is a multiple of the first, which means that there is a linear relation among the rows, which means that the conic is degenerate.  Thus, this case is impossible, which shows that both of the two conics must be tangent to the line of intersection of the two planes.

So, suppose that the two conics are both tangent to the line of intersection of the two planes.  Up to a choice of coordinates on $\PP^n$, we can assume our two conics have the form
\[ (s^2,st,t^2,0,0,\dots,0) \]
and
\[ (0, g_1, g_2, g_3, 0, \dots, 0).\]
As before, write $g_i = a_i s^2+ b_i st + c_i t^2$.  Since the second conic is also tangent to the line of intersection of the two planes, we see that $g_3$ is a square, i.e., $g_3 = (us+vt)^2$.  Then we see that the fibers are tuples $(f_0, f_1, \dots, f_n)$ where
\[ \left[ \begin{array}{llll}
2s & t & 0 & 0 \\
0 & s & 2t & 0 \\
0 & 2a_1s + b_1t & 2a_2s + b_2t & 2u(us+vt) \\
0 & b_1s+2c_1t & b_2s+2c_2t & 2v(us+vt) 
\end{array} \right]
\left[ \begin{array}{l}
f_0 \\
f_1 \\
f_2 \\
f_3
\end{array}
\right]
=
\left[ \begin{array}{l}
0 \\
0 \\
0 \\
0
\end{array}
\right]
\]

Taking $-v$ times the third row of the matrix plus $u$ times the fourth row of the matrix gives
\[ \ell_1 f_1 + \ell_2 f_2 = 0 \]
where $\ell_1 = -v(2a_1s+b_1t) + u(b_1s+2c_1t)$ and $\ell_2 = -v(2a_2s + b_2 t)+u(b_2s+2c_2t)$.  Since we also have
\[ sf_1 + 2t f_2 = 0 \]
we see that either $f_2 = 0$ or the two conics satisfy the parameterized tangency relation.  

To compute the fiber dimension of $\pi_1$, observe that $f_0$ determines $f_1, f_2, f_3$ and must be divisible by $t^3$. Hence, in total there are $3e+6$ conditions on the fiber of $\pi_1$. This concludes the proof.
\end{proof}

\begin{proof} (Theorem \ref{twoConicClassification}) We have already seen that the component $\mathcal{G} = \pi_2(\Gamma) \cap \M(b_{\bullet}(2^2))$ is an irreducible component. It follows easily from  Lemma \ref{paramTangencyLemma} and Corollary \ref{twoConicsCor} that there is at most one more component of $\M(b_{\bullet}(2^2))$ corresponding to the locus $\mathcal{PT}$ corresponding to conic relations that satisfy the parameterized tangency condition. We later exhibit an element of $\mathcal{PT}$ to show that it is nonempty. If  $\dim(\mathcal{PT} ) \geq \dim(\mathcal{G})$, then $\mathcal{PT}$ is a separate component from $\mathcal{G}$. We now compute $\dim(\mathcal{PT})$. To choose the first conic, we need to specify the plane spanned by the conic and three degree $2$ polynomials mapping to that plane. The dimension of the Grassmannian $\mathbb{G}(2,n)$ is $3 (n-2)$. Hence, there are $3(n-2)+9$ dimensions of choice for the first unscaled parameterized conic $C_1$.  Then, there is a $1$-dimensional family of tangent lines to $C_1$, given by the vanishing of some coordinate $v$ on $C_1$. Given a tangent line $\ell$, the set of planes $\Lambda$ that contain $\ell$ is a Schubert variety in $\mathbb{G}(2,n)$ of dimension $n-2$.  Finally, there is a $5$-dimensional family of unscaled parameterized conics $a$ satisfying the parameterized tangency condition with respect to $C_1$. To see this, note that there is a $2$-dimensional choice of coordinate $v$ such that $\partial_v C_2 = \partial_u C_1$ (Note that if $C_2$ is to satisfy the parameterized tangency condition, we can always choose such $v$. We could alternatively compute the dimension by allowing $\partial_v C_2$ to be merely a scalar multiple of $\partial_u C_1$; in that case we would only have a $1$-dimensional choice of $v$, since it would be only defined up to scaling, but we would then be able to scale $C_1$ by a $1$-dimensional choice of parameter).   Given $\partial_v C_2$, there is a $3$-dimensional family of possible $C_2$.  This gives a $3(n-2)+9+1+n-2+5 = 4n +7$ dimensional family of ordered pairs of conics satisfying the parameterized tangency condition.  The dimension of the fibers of $\pi_1$ over this locus is $(e+1)(n+1) - 3e-7$, and the dimension of the fibers of $\pi_2$ over this locus in $\aA$ is $4$, showing that $\mathcal{PT}$ has dimension $4n+7+(e+1)(n+1)-3e-7-4 = e(n-2)+5n-3$.  This is at least the dimension of $\mathcal{G}$ when $e \geq 2n-3$ by Proposition \ref{goodComp}.

Finally, we exhibit an example of a curve in $\mathcal{PT}$ using a construction similar to the one in the proof of Proposition \ref{goodComp}.  Express $2e-6 = (n-3) q +r$. 
\begin{itemize}
\item If $n$ is odd, then $r$ is even.  Consider the sequence
\begin{equation}\label{seq3}
1,1,1,x_1, 1, x_2, 1, \dots, x_{\frac{n-3}{2}}, 1,
\end{equation}
where $x_1= \dots = x_{\frac{r}{2}}= q$ and $x_{\frac{r}{2}+1} = \dots = x_{\frac{n-3}{2}}= q-1$.  Let $k_0=e$. For $1 \leq i \leq n$, let $k_i$ be defined by the property that $k_{i-1} - k_{i}$ is the $i$th entry in Sequence  (\ref{seq3}). Observe that $k_{n-1} =1$ and $k_n =0$. Let
\[ f = (s^{k_0}, s^{k_1}t^{e-k_1}, s^{k_2}t^{e-k_2}, \dots, s^{k_{n-1}} t^{e-k_{n-1}}, t^{e-k_n}). \]
Write $N_f \cong \bigoplus_{i=1}^{n-1} \OO(e+b_i)$.  By Theorem \ref{monomialSplitting},  $b_i = k_{i-1} - k_{i+1} = (k_{i-1} - k_i) + (k_i - k_{i+1})$, which is the sum of the $i$th and $(i+1)$st entries in the sequence. Hence, $b_1=b_2=2$. Moreover, $r$ of the $b$'s are equal to $q+1$ and the rest are equal to $q$. Hence, $N_f$  has the required form. Notice that the first $4$ coordinates of $f$ have the form $(s^e, s^{e-1}t, s^{e-2}t^2, s^{e-3}t^3, \dots, )$. The first four columns $R_1, \dots, R_4$ of $\partial f$ satisfy the two degree two relations $$t^2 R_1 - 2st R_2 + s^2 R_3=0, \quad t^2 R_2 - 2st R_3 + s^2 R_4=0.$$ These relations satisfy the parameterized tangency condition. 
\item If $n$ is even, then let the sequence be
\begin{equation}\label{seq4}
1,1,1,x_1, 1, x_2, 1, \dots, x_{\frac{n-2}{2}-1}, 1, 
\end{equation}
where $x_1= \dots = x_{\lfloor \frac{r}{2} \rfloor}= q$ and $x_{\lfloor \frac{r}{2} \rfloor+1} = \dots = x_{\frac{n-2}{2}-1}= q-1$.  We invoke Corollary \ref{cor-examples} with $\delta_i$ given by Sequence (\ref{seq4}). Then $b_i = \delta_i + \delta_{i+1}$ for $1 \leq i \leq n-2$. Hence, $b_1 = b_2 =2$. If $1 \leq i \leq n-2$, $2 \lfloor \frac{r}{2} \rfloor$ of the $b_i$ are equal to $q+1$, the rest are $q$. If $r$ is even, then $b_{n-1}= q$. Otherwise, $b_{n-1} = q+1$. Corollary \ref{cor-examples} applies and provides a curve with the desired normal bundle. This curve has two degree two relations that satisfy the parameterized tangency condition as in the previous case. 

\end{itemize}

\end{proof}

\subsection{More conics}
In this section, by considering chains of conic relations that satisfy the parameterized tangency condition, we will show that the number of components of $\M(b_{\bullet}(2^k))$ grows at least linearly with $k$ for sufficiently large $e$ and $n$.

Let $B_j^k$ denote the ordered $k$-tuple of unscaled parameterized conics $(C_1, \dots, C_k)$, where $C_i$ with $i \leq j$ is the unscaled parameterized conic $(s^2, -2st, t^2)$ contained in the plane $x_h=0$ for $h \not= i-1,i, i+1$ and $C_i$ with $i>j$ is the unscaled parameterized conic $(s^2, -2st, t^2)$ contained in the plane $x_h =0$ for $h \not=3 i -2j -1, 3i-2j, 3i-2j+1$. For example, $B_3^5$ is the following tuple of unscaled parameterized conics
$$\begin{array}{cccccccccccccc} (s^2, & -2st, & t^2,&  0, &0, &0, &0, &0, &0 , &0, &0, &0, &\dots &0) \\ (0, &s^2, &-2st, & t^2, & 0, & 0, &0, & 0, & 0, &0, &0, &0, &\dots &0) \\ (0,& 0, & s^2, & -2st, & t^2, & 0, & 0, & 0,& 0,& 0, & 0,& 0, &\dots & 0) \\ (0, & 0, & 0, & 0, & 0, & s^2, & -2st, & t^2, & 0, & 0, & 0,& 0, &\dots & 0)  \\ (0,& 0,& 0, & 0,& 0, & 0, & 0, & 0,  & s^2, & -2st, & t^2, & 0, &\dots & 0) \end{array}$$
The conics $C_i$ and $C_{i+1}$ in $B_j^k$ satisfy the parameterized tangency condition for $1\leq i <j$ and the rest of the conics are general.

\begin{theorem}\label{thm-moreconics}
Let $n \geq 3k-1$ and $e> 2 k n- 2n -2$. Then  $\M(b_\bullet(2^k))$ has at least $k$ components.
\end{theorem}
\begin{proof}
Let $\mathcal{C}_j^k$ denote the locus of unscaled parameterized conics $(C_1, \dots, C_k)$ in $\mathcal{C}^k$ such that $C_i$ and $C_{i+1}$ satisfy the parameterized tangency condition for $1 \leq i < j$ and $C_i$ are general for $i > j$.  In particular, $B_j^k \in \mathcal{C}_j^k$. We first compute the dimension of the locus $\mathcal{C}_j^k$. As in the proof of Proposition \ref{goodComp}, a general unscaled parameterized conic depends on $3(n-2) + 9$ parameters ($3(n-2)$ for the choice of the plane of the conic and $9$ for the three degree 2 polynomials into this plane).  Given a conic $C_1$ there is an $n+4$ parameter family of conics satisfying the parameterized tangency condition with respect to $C_1$ (There is a 1 parameter family of tangent lines to $C_1$. Given a fixed tangent line $l$ there is an $(n-2)$-dimensional family of planes containing $l$. Then there is a 5-dimensional space of unscaled parameterized conics with the given data as in the proof of Theorem \ref{twoConicClassification}.) Consequently, \begin{align*}\dim(\mathcal{C}_j^k) &= (k-j+1)(3(n-2)+9) + (j-1)(n+4) & \\  &= k(3n+3) - (j-1)(2n-1).&\end{align*} 

For a general point $x=(C_1, \dots, C_k) \in \mathcal{C}_j^k$, the set of partial derivatives of the $C_i$ span a linear space of dimension $2k-j+1$, since there are $2k$ partial derivatives, and $j-1$ relations that overlap from the parameterized tangency conditions. Let $p= (p_0, \dots, p_n)$ be a point of the span of these partial derivatives. Then the relation $\sum p_i f_i=0$ is a polynomial in $s,t$ of degree $e+1$, hence imposes at most $(e+2)$ conditions on the fibers of $\pi_1$. We conclude that the dimension of $\pi_1^{-1}(x)$ is at least $(e+1)(n+1)-1 - (2k-j+1)(e+2)$. 

On the other hand, consider the  dimension of  $\pi_1^{-1}(B_j^k)$. The $f_i$ satisfy the relations $$f_i = \left(\frac{s}{t}\right)^i f_0\quad \mbox{for} \ \  1 \leq i \leq j+1$$ and $$f_{3i-2j} = \frac{s}{t} f_{3i-2j-1}, \quad f_{3i-2j+1} = \left(\frac{s}{t}\right)^2 f_{3i-2j-1} \quad \mbox{for}  \  j<i\leq k.$$ Hence, $f_0$ can be chosen freely subject to the condition that it is divisible by $t^{j+1}$. This determines $f_i$ for $1 \leq i \leq j+1$. Then for $j < i \leq k$, the entry $f_{3i-2j-1}$ can be chosen freely subject to the condition that it is divisible by $t^2$. This determines $f_{3i-2j}$ and $f_{3i-2j+1}$. All remaining $f_i$ are free. We conclude that $$\dim(\pi_1^{-1}(B_j^k)) = (e+1)(n+1)-1 - (2k-j+1)(e+2).$$ Hence, the general fiber of $\pi_1$ over $\mathcal{C}_j^k$ is irreducible of dimension $(e+1)(n+1)-1 - (2k-j+1)(e+2).$ We conclude that there is a component $\mathcal{V}_j$ of $\pi_1^{-1}(\mathcal{C}_j^k)$  with  $$\dim(\mathcal{V}_j) = (e+1)(n+1)-1 + k(3n-2e-1) + (j-1)(e-2n+3).$$  We warn the reader that $\mathcal{V}_j$  is typically  not a component of the incidence correspondence. However, we will shortly show that each $\mathcal{V}_j$ has to be contained in a distinct irreducible component $\mathcal{P}_j$ of the incidence correspondence $\pi_1^{-1}(\mathcal{C}^k)$.

Suppose there exists an irreducible component $U$ containing $\mathcal{V}_{j_1}$ and $\mathcal{V}_{j_2}$ for $j_1 < j_2$. Then $\pi_1(U)$ contains $B_{j_1}^k$. Hence, the general fiber dimension of $\pi_1$ restricted to $U$ is at most $(e+1)(n+1)-1 - (2k-j_1+1)(e+2).$ Hence, the dimension of $U$ is at most $\dim(\mathcal{C}^k) + (e+1)(n+1) -1- (2k-j_1+1)(e+2).$ However, the dimension of $\mathcal{V}_{j_2}$ is $(e+1)(n+1)-1 + k(3n-2e-1) + (j_2-1)(e-2n+3)$. We bound $$\dim(\mathcal{V}_{j_2}) - \dim(U) \geq (j_2-j_1)e + 3j_2 -2j_1 - 2j_2 n +2n -1 \geq e + 2n +2 - 2 kn.$$ By our assumption on $e$, this number is positive. This is a contradiction. We conclude that $\mathcal{V}_j$ belong to different components for each $1 \leq j \leq k$.  

Now consider the projection $\pi_2(\mathcal{V}_j)$. We will shortly see that the general member of $\mathcal{V}_j$ has the desired normal bundle. Consequently, there are exactly $k$ independent conic relations among the rows of $\partial f$ and the general fiber dimension of $\pi_2$ restricted to $\mathcal{P}_j$ is $k^2$. In fact, $\pi_2$ is generically a $\GL(k)$-bundle corresponding to choices of bases for the conic relations among the columns of $\partial f$. Consequently, $\pi_2(\mathcal{P}_j)$ is a distinct irreducible component of $\M(b_{\bullet}(2^k))$ for each $1 \leq j \leq k$. 

Finally, using Sacchiero's construction, we see that there are unramified maps in these loci that lie in $M(b_{\bullet}(2^k))$. Using the division algorithm, write $2e-2 - 2k = q(n-k-1) + r$ with $0 \leq r < n-k-1$. We construct a curve with $$N_{f} = \OO(e+2)^k \oplus \OO(e+q)^{n-k-1-r} \oplus \OO(e+q+1)^r.$$  The construction depends on whether $n-k$ is odd or even.
\begin{itemize}
\item If $n-k$ is odd, then define the sequence
\begin{equation}\label{seqa1}
1,1,\dots,1,1,x_1, 1,1,x_2, \dots, 1, 1, x_{k-j}, 1, x_{k-j+1}, 1, x_{k-j+2} , \dots, x_{\frac{n-k-1}{2}}, 1,
\end{equation}
 where there are $j+1$ $1$'s at the beginning, $x_1 = \dots =x_{\frac{r}{2}} = q$ and $x_{\frac{r}{2}+1} = \dots =  x_{\frac{n-k-1}{2}} = q-1$.  Set $k_0 = e$ and for $1 \leq i \leq n$, let $k_{i-1} - k_{i}$ equal to the $i$th entry of  Sequence (\ref{seqa1}).   Let $f$ be the monomial map 
$$f = (s^{k_0}, s^{k_1}t^{e-k_1}, s^{k_2}t^{e-k_2}, \dots, s^{k_{n-1}} t^{e-k_{n-1}}, t^{e-k_n}) $$
Write $N_f \cong \bigoplus_{i=1}^{n-1} \OO(e+b_i)$.  By Theorem \ref{monomialSplitting}, the $b_i$ are given by $k_{i-1} - k_{i+1} = k_{i-1} - k_i + k_i - k_{i+1}$, which is the sum of the $i$th and $(i+1)$st entries in the sequence. Hence the $b_i$ have the required form. Furthermore, by the proof of Theorem \ref{monomialSplitting}, $\partial f$ satisfies the desired relations.

\item If $n-k$ is even, then define the sequence
 \begin{equation}\label{seqa2}
  1,1, \dots, 1, 1, x_1, 1, 1, x_2, \dots, 1, 1, x_{k-j}, 1, x_{k-j+1}, 1, x_{k-j+2}, \dots, x_{\frac{n-k}{2}-1}, 1,
 \end{equation}
where there are $j+1$ $1$'s at the beginning, $x_1 = \dots =x_{\lfloor \frac{r}{2} \rfloor} = q$ and $x_{\lfloor \frac{r}{2} \rfloor+1 } = \dots =  x_{\frac{n-k}{2}-1} = q-1$.  We invoke Corollary \ref{cor-examples} with $\delta_i$ given by Sequence (\ref{seqa2}). As in the proof of Theorem \ref{twoConicClassification}, the hypotheses of  Corollary \ref{cor-examples} hold and there is a curve with the required form and relations.

\end{itemize}

\end{proof}

The above proof suggests a way to possibly get many more irreducible components of $\M(b_{\bullet}(2^k)$. Let $\mu$ be a partition of $k$ with $h$ parts $k= k_1 + k_2 \dots + k_h$. Let $\mathcal{C}_{\mu}$ be the locus of unscaled parameterized conics $(C_1, \dots, C_k)$, where consecutive $C_i$ satisfy the parameterized tangency condition according to whether their indices are in the same part of the partition $\mu$. In other words $C_i$ and $C_{i+1}$ satisfy the parameterized tangency condition for any index $i$ with  $\sum_{j=1}^l k_j < i <  \sum_{j=1}^{l+1} k_j$ for some $0 \leq l < h$. Let $\mathcal{P}_{\mu} =  \pi_1^{-1}(\mathcal{C}_{\lambda}^k)$. By an argument identical to that of Theorem \ref{thm-moreconics}, if $\mu$ and $\nu$ have different numbers of parts, the loci $\mathcal{P}_{\mu}$ and $\mathcal{P}_{\nu}$ belong to different components if $e$ is sufficiently large. However, if $\mu$ and $\nu$ have the same number of parts, it is possible that $\mathcal{P}_{\mu}$ and $\mathcal{P}_{\nu}$ could both be in the closure of another larger component. We pose the following natural question.

\begin{question}
Let $\mu$ and $\nu$ be two different partitions of $k$. Do $\mathcal{P}_{\mu}$ and $\mathcal{P}_{\nu}$ belong to different irreducible components of the incidence correspondence?
\end{question}

\begin{remark}
Since it is possible to construct elements in $\mathcal{P}_{\mu}$ that map to $\M(b_{\bullet}(2^k))$ under $\pi_2$, a positive answer to the question would imply that the number of irreducible components of $\M(b_{\bullet}(2^k))$ is at least the number of partitions of $k$ provided that $n \geq 3k-1$ and $e$ is sufficiently large. This would  provide superpolynomial growth for the number of components.
\end{remark}

\section{Higher degree relations}\label{sec-higherdegree}
In this section, we generalize the discussion for $\M(b_{\bullet}(2^k))$ to $\M(b_{\bullet}(d^k))$ and for $n\geq 5$ exhibit multiple irreducible components of  $\M(b_{\bullet}(d^k))$. 

\begin{theorem}
\label{higherDegreeRelations}
For $n \geq 3k-1$ and $k \geq 2$ even, $\M(b_{\bullet}(d^k))$ has at least $\frac{k}{2} + 1$ components for $e \geq k(d+1)(n+1)$.
\end{theorem}
\begin{proof}
Let $\mathcal{D}^k$ be the space of ordered $k$-tuples of independent unscaled parameterized degree $d$ rational curves.  Let $I_{k,e}$ be the incidence correspondence parameterizing pairs $(D, f)$, where $f$ is a degree $e$ rational curve and $D\in \mathcal{D}^k$ is a set of $k$ independent degree $d$ relations among the columns of $\partial f$. Let $\pi_1$ and $\pi_2$ denote the two projections to $\mathcal{D}^k$ and $\Mor_e(\PP^1, \PP^n)$, respectively.  We find $\frac{k}{2} + 1$ components of $\M(b_{\bullet}(d^k))$.  

First, for $0 \leq j \leq \frac{k}{2}$, we construct an element $B_j = (B_{j,1}, \dots, B_{j,k}) \in \mathcal{D}^k$, where for $1 \leq i \leq j$, $B_{j,2i-1}$ and $B_{j,2i}$ are given by
$$\begin{array}{rrrrrr} ( \dots 0, & (d-1)t^d, &-dst^{d-1}, & s^d, & 0, & 0 \dots ) \\ ( \dots 0, & 0, & t^d, &  -ds^{d-1}t, & (d-1)s^d, & 0 \dots ) \end{array},$$
where the nonzero coordinates are $x_u$ for $4i-4 \leq u \leq 4i-1$.  For $i > j$, we let $B_{j,2i-1}$ and $B_{j,2i}$ be given by
$$\begin{array}{rrrrrrrr} (\dots 0, & t^d, &-ds^{d-1}t, &(d-1)s^d,& 0,& 0, &0 , & 0 \dots ) \\ (\dots 0, & 0, & 0,&  0, & (d-1)t^d,&  -dst^{d-1}, & s^d,& 0 \dots )\end{array}, $$
where the nonzero coordinates are $x_u$ for $6i-2j-6 \leq u \leq 6i-2j-1$.

We work out the dimensions of the fibers of $\pi_1$ over $B_j$.  It is clear from the definitions that for $i_1 \neq i_2$, the conditions imposed on the fibers of $\pi_1$ by the pair $B_{j,2i_1-1}, B_{j,2i_1}$ and the pair $B_{j,2i_2-1}, B_{j,2i_2}$ are independent.  For $i \leq j$, the matrix of partial derivatives of $B_{j,2i-1}$ and $B_{j,2i}$ is
\[ A = \left[ \begin{array}{llll}
d(d-1)t^{d-1} & -d(d-1)st^{d-2} & 0 & 0 \\
0 & -d t^{d-1} & ds^{d-1} & 0 \\
0 & d t^{d-1} & -ds^{d-1} & 0 \\
0 & 0 & -d(d-1)s^{d-2}t & d(d-1) s^{d-1}
\end{array} \right] .\]
From the relations 
\[ A \left[ \begin{array}{l} f_{4i-4} \\ f_{4i-3} \\ f_{4i-2} \\ f_{4i-1}
\end{array} \right] = 0\]
we see that $s^{d+1}$ divides $f_{4i-4}$, but that subject to that condition, $f_{4i-4}$ can be chosen freely and this choice completely determines $f_{4i-3}, f_{4i-2}$, and $f_{4i-1}$.  This makes for $3(e+1) + d+1 = 3e+d+4$ conditions.  For $i > j$, we see by a similar calculation that $B_{j,2i-1}$ and $B_{j,2i}$ impose $2(e+d+e+2) = 4e+2d+4$ conditions.  Thus, the total number of conditions imposed is $j(3e+d+4)+(k-2j)(2e+d+2) = k(2e+d+2)-j(e+d)$.

For future use, we construct an example of a curve in the fiber $\pi_1^{-1}(B_j)$ corresponding to an unramified nondegenerate map that lies in $\M(b_{\bullet}(d^k))$.  Let $q$ and $r$ be defined by $2e-2-kd = q(n-1-k)+r$ with $0 \leq r < q$.  Consider the following sequence $\delta_1, \dots, \delta_{n-1}$.  For $1 \leq i \leq 4j$ the sequence repeats the length four pattern $1, d-1, 1, x_l$ and looks like
\[1, \ d-1, \ 1, \ x_1, \ 1, \ d-1, \ 1, \ x_2, \ \dots \ 1, \ d-1, \ 1, \ x_j .\]
For $ 4j+1 \leq i \leq 3k-2j$ the sequence repeats the length six pattern
\[ 1, \ d-1, \ x_l - (d-2), \ d-1, \ 1, \ x_{l+1} \]
and looks like
\[ 1, \ d-1, \ x_{j+1} - (d-2), \ d-1, \ 1, \ x_{j+2}, \dots, \ 1, \ d-1, \ x_{k-j-1} - (d-2), \ d-1, \ 1, \ x_{k-j} . \]
Finally, for $3k-2j+1 \leq i \leq n-1$, the sequence repeats the length two pattern $1, x_{l}$ and looks like
$$1, \ x_{k-j+1}, \ 1, \  x_{k-j+2}, \dots $$
Here, $x_1 = \dots = x_{\lfloor \frac{r}{2} \rfloor} = q$ and $x_{\lfloor \frac{r}{2} \rfloor + 1} = \dots = x_{\lfloor \frac{n-k-1}{2} \rfloor} = q-1$.  For example, if $n=19$, $e=41$, $d=3$, $k=6$, and $j=2$, we have $q = 5$ and $r=2$ and we have the sequence
\[ 1, 2, 1, 5 , 1,2,1,4, 1, 2, 3, 2, 1, 4, 1, 4, 1, 4 . \]
By Corollary \ref{cor-examples}, there exists a curve with the required normal bundle and relations.

We now argue that for $j_1 < j_2$, $\pi_1^{-1}(B_{j_1})$ and $\pi_1^{-1}(B_{j_2})$ lie in different irreducible components of $\pi_1^{-1}(\mathcal{D}^k)$.  To get a contradiction, suppose that some component $U$ of $\pi_1^{-1}(\mathcal{D}^k)$ contains both $\pi_1^{-1}(B_{j_1})$ and $\pi_1^{-1}(B_{j_2})$.  Then since $\pi_1^{-1}(B_{j_1})$ lies in $U$, the general fiber of $\pi_1$ restricted to $U$ has dimension at most $(e+1)(n+1)-1 - k(2e+d+2)+j_1(e+d)$, which shows that the dimension of $U$ is at most $\dim \mathcal{D}^k + (e+1)(n+1) -1- k(2e+d+2)+j_1(e+d)$.  However, the dimension of $\pi_1^{-1}(B_{j_2})$ is at least $(e+1)(n+1) -1- k(2e+d+2)+j_2(e+d)$.  Bounding $\pi_1^{-1}(B_{j_2}) - \dim U$, we get
\[  \pi_1^{-1}(B_{j_2}) - \dim U  \geq (j_2-j_1)(e+d) - k(d+1)(n+1) \geq e+d - k(d+1)(n+1) > 0 \]
by our assumption on $e$.

The examples show that any irreducible component of the incidence correspondence containing a $B_j$ is generically a $\GL(k)$ bundle over its image in $\pi_2$, so $\pi_2(\pi_1^{-1}(B_j))$ all lie in different components for each $0 \leq j \leq \frac{k}{2}$.  The result follows.

\end{proof}

\section{Examples}\label{sec-examples}
In this section, we discuss some basic examples of strata of rational curves with fixed normal bundle.  We construct an example of a stratum of rational curves in $\PP^4$  with higher than expected dimension.  We find examples of reducible strata of curves in $\PP^5$.  We show that a natural generalization of the example of Alzati and Re \cite{AlzatiRe1} has at least three reducible components.  Finally, we provide an example of reducible strata $M_e(b_{\bullet})$ with $b_1 \neq b_2$.

\subsection{Conics in $\PP^4$}
Many of the results in section \ref{sec-conics} were only for $n \geq 5$.  In this section we completely describe $\M(b_{\bullet}(2^2))$ for degree $e$ curves in $\PP^4$.

\begin{proposition}
For $e \geq 5$ and $n=4$, $\M(b_{\bullet}(2^2))$ is irreducible of dimension $2e+18$.  This is larger than the expected dimension for $e \geq 6$.
\end{proposition}
\begin{proof}
First we show that if $f$ is a degree $e$ rational curve in $\PP^4$ such that the relations among the columns of $\partial f$ correspond to two conics in the dual space whose planes meet in a point, then $f$ is degenerate.  To see this, note that for two such conics, their partial derivatives span a $4$-dimensional vector space of degree $1$ maps to $\PP^{4*}$.  By Lemma \ref{lem-basicpolynomial} the partial derivatives give a $4$-dimensional space of linear forms $a_i$ such that $\sum_{j} a_{ij}f_j = 0$.  This shows that if $f$ were nondegenerate (which would imply $f^* T_{\PP^n}$ contains no $\OO(e)$ factors), the restricted tangent bundle $f^{*}T_{\PP^n}$ would be $\OO(e+1)^4$, which is impossible by degree considerations.  Thus, any such $f$ must be degenerate.

There is another component, however, corresponding to pairs of conics satisfying the parameterized tangency condition by Lemma \ref{paramTangencyLemma}.  We compute the dimension of this locus. The dimension of the space of unscaled parameterized conics in $\PP^4$ is $3(n-2)+9 = 15$.  Given the first conic, there is a $1$-dimensional choice of tangent lines, then an $2$-dimensional family of planes containing this tangent line, followed by a $5$-dimensional family of conics satisfying the parameterized tangency condition, for a total of $8$ dimensions.  Thus, this corresponds to a $23$-dimensional locus in $\mathcal{C}^2$.  The fiber of $\pi_1$ over this locus is $(e+1)(n+1) - 3(e+2) = 2e-1$-dimensional.  The fibers of $\pi_2$ over this locus are $4$-dimensional, so the dimension of this family is $2e-1 + 23 - 4 = 2e+18$.  The expected dimension of $\M(b_{\bullet}(2^2))$ is $5(e+1) - (4e-18) = e+23$, so we see that for $e \geq 6$, this has larger than expected dimension.
\end{proof}
\subsection{An example in $\PP^5$}
We can find the smallest example where $\M(b_{\bullet}(2^2))$ has two components.  In particular, note that both $e$ and $n$ are smaller than the $e=11$, $n=8$ example discovered by Alzati and Re.

\begin{corollary}
The space $\M(b_{\bullet}(2^2))$ has two components for $e \geq 2n-3$, $n \geq 5$.  In particular for $n = 5$, $e = 7$, $\M(b_{\bullet}(2^2))$ is reducible.
\end{corollary}
\begin{proof}
This follows directly from Theorem \ref{twoConicClassification}.
\end{proof}

\begin{remark}
The normal bundle of curves in $\M(b_{\bullet}(2^2))$ for $e = 7$, $n=5$ is $\OO(9)^2 \oplus \OO(11)^2$, so the expected codimension is $4$.
\end{remark}

Thus, we see that as soon as $n > 4$, we immediately start getting reducible strata.

\subsection{Alzati and Re's example} Alzati and Re \cite{AlzatiRe1} exhibit two distinct irreducible components of the locus in $\Mor_{11}(\PP^1, \PP^8)$ where the normal bundle is $\OO(13)^3\oplus \OO(14)^2 \oplus \OO(15)^2$. Theorem \ref{thm-moreconics} also implies the existence of these components. In their example, in one component the conic relations are general. In the other component, two of the conic relations satisfy the parameterized tangency condition. In fact, by increasing the degree ($e > 30$ certainly suffices), one can obtain examples with more than two components. 

\subsection{An example without duplicate lowest factors}
Finally, we work out examples of reducible $M_e(b_{\bullet})$ where the two lowest $b_i$ are distinct.

\begin{theorem}
Let $d_2 \geq d_1 \geq 2$ be integers, $e \geq (n+1)(d_1+d_2+2)-d_1$ and $n \geq 5$.  Let $q$ and $r$ be defined by $2e-2-d_1-d_2 = q(n-3)+r$, and let $b_1 = d_1$, $b_2 = d_2$, $b_3 = \dots = b_{n-r-1} = q$, $b_{n-r} = \dots = b_{n-1} = q+1$.  Then $\M(b_{\bullet})$ is reducible.
\end{theorem}
\begin{proof}
We exhibit two components of the incidence correspondence $\aA$ consisting of the set of tuples $(a_1, a_2, f)$ where $a_1$ and $a_2$ are unscaled parameterized curves of degree $d_1$ and $d_2$, $f \in \M(b_{\bullet})$ and $\sum_{i=0}^n a_{ji} \partial_{\ell} f_i = 0$ for $j \in \{1, 2\}$ and $\ell \in \{s, t\}$.

We start by finding a component of dimension close to the expected dimension.  Consider the following sequence
\[ 1, \ d_1-1, \ x_1, \ d_1-1, \ d_2-d_1+1, \ x_2, \ d_2-d_1+1, \ x_3, \ d_2-d_1+1, \ x_4, \ \dots \]
where 
$$ x_1 = q-d_1+1, \  x_2 = \dots = x_{\lfloor \frac{r}{2} \rfloor + 1} = q-d_2+d_1 $$ and the rest of the $x_i$ are $q-d_2+d_1-1$.  By Corollary \ref{cor-examples}, there is an unramified map $f$ with normal bundle $N_f = \OO(e+d_1) \oplus \OO(e+d_2) \oplus \OO(e+q+1)^r \oplus \OO(e+q)^{n-3-r}$. Furthermore, this map $f$ satisfies  syzygies  of  degrees $d_1$ and $d_2$ with respect to 
$$
a_1 = ((d_1-1)t^{d_1}, \ -d_1st^{d_1-1}, \ s^{d_1}, 0,  0,  0, \dots)$$
and
$$a_2 = (0, 0, 0, (d_2-d_1+1)t^{d_2}, \ -d_2 s^{d_1-1}t^{d_2-d_1+1},  \ (d_1-1)s^{d_2}, 0, \dots).$$
Hence, $f \in \pi_1^{-1}(a_1, a_2)$. 
We claim the dimension of $\pi_1^{-1}(a_1, a_2)$ is $(e+1)(n+1)-4e-d_1-d_2-5$. The relations $\sum_{i=0}^2 f_i \partial_s a_{1i} = 0 =  \sum_{i=0}^2 f_i \partial_t a_{1i} $ imply that $s^{d_1-1}t | f_1$ and $f_1$ determines $f_0$ and $f_2$. Similarly, the relations $\sum_{i=3}^5 f_i \partial_s a_{2i} = 0 =  \sum_{i=3}^5 f_i \partial_t a_{2i}$ imply that $s^{d_2-d_1+1}t^{d_1-1} | f_4$ and $f_4$ determines $f_3$ and $f_5$. This yields the desired fiber dimension. Since the dimension of the space of ordered pairs of unscaled parameterized curves of degrees $d_1$ and $d_2$ is at most $(n+1)(d_1+d_2+2)$, we conclude that there is a component of $\aA$ of dimension at most $(n+1)(d_1+d_2+e+3) - 4e - (d_1+d_2)-5$.

Now we show there is another component of dimension at least as large.  To show this, we need only find one pair of relations $(a_1,a_2)$ of degrees $d_1$ and $d_2$ such that the space of $f$ satisfying $\sum_{i=0}^n a_{ji} \partial_{\ell} f_i = 0$ for $j \in \{1,2\}$ and $\ell \in \{s, t\}$ contains nondegenerate unramified maps and has codimension at most $3e + d_2 + 4$. Then the dimension of this component is at least $$(n+1)(e+1) -5 - 3e - d_2.$$  A simple check shows that the inequality $e \geq (d_1+d_2+2)(n+1)-d_1$ guarantees that this dimension is at least the dimension of the previous component.  To find such an example, consider the relations
$$ ((d_1-1)t^{d_1}, \ -d_1st^{d_1-1}, \ s^{d_1}, 0, 0, \dots, 0) $$ and 
$$ (0,  (d_2-d_1+1)t^{d_2},  \ -d_2s^{d_1-1}t^{d_2-d_1+1},  \ (d_1-1)s^{d_2}, 0, \dots, 0) . $$
We see that the conditions imposed on the fiber over these relations are
\[ \left[ \begin{array}{llll}
0 & -d_1t^{d_1-1} & d_1 s^{d_1-1} & 0 \\
d_1(d_1-1)t^{d_1-1} & - d_1(d_1-1)st^{d_1-2} & 0 & 0 \\
0 & 0 & -d_2(d_1-1)s^{d_1-2}t^{d_2-d_1+1} & d_2(d_1-1)s^{d_2-1} \\
0 & d_2(d_2-d_1+1)t^{d_2-1} & -d_2(d_2-d_1+1)s^{d_1-1}t^{d_2-d_1} & 0
\end{array} \right] 
\left[ \begin{array}{l}
f_0 \\ f_1 \\ f_2 \\ f_3
\end{array} \right] = 0 .\]
The last row is a multiple of the first, and hence it imposes no new conditions on the $f_i$. Note that knowing $f_1$ determines $f_0$ and $f_2$; and knowing $f_2$ determines $f_3$. The relations imply that $s^{d_2} t $ divides $f_1$, otherwise $f_1$ is free. All other $f_i$ are free. Hence, the fiber has dimension $$(e+1)(n+1) -3e-d_2-5.$$ 

It remains to find a single example of a nondegenerate unramified map and these relations.  Let $q$ and $r$ be defined by $2e-2-d_1-d_2 = q(n-3) + r$.  Consider the following sequence
\[ 1, \ d_1-1, \ d_2-d_1+1, \ x_1, \ d_2-d_1+1, \ x_2, \ d_2-d_1+1, \ x_3,  \ d_2-d_1+1, \ \dots \]
where $x_1 = \dots = x_{\lfloor \frac{r}{2} \rfloor} = q-d_2+d_1$ and $x_{\lfloor \frac{r}{2} \rfloor+1} = \dots = q-d_2+d_1-1$.  Then by Corollary \ref{cor-examples}, there exists a curve of the required form having the required relations.
\end{proof}

\bibliographystyle{plain}

\end{document}